\documentclass{article}

\usepackage{arxiv}

\usepackage[utf8]{inputenc} 
\usepackage[T1]{fontenc}    
\usepackage[hidelinks]{hyperref}
\usepackage{graphicx}
\usepackage{caption}
\usepackage{subcaption}
\usepackage{url}            
\usepackage{booktabs}       
\usepackage{amsfonts}       
\usepackage{nicefrac}       
\usepackage{microtype}      
\usepackage{lipsum}		
\usepackage[round]{natbib}
\usepackage{doi}
\usepackage{amsthm}
\usepackage{amsmath}
\usepackage{mathabx}
\usepackage{diagbox}
\usepackage{stackengine}
\usepackage{algorithm}
\usepackage{algpseudocode}
 \usepackage{setspace}
\let\Algorithm\algorithm
\renewcommand\algorithm[1][]{\Algorithm[#1]\setstretch{1.2}}

\usepackage{array}
\usepackage{eqparbox}

\newtheorem{thm}{Theorem}

\title{Moment Monotonicity of Weibull, Gamma and Log-normal Distributions}

\date{} 					

\author{ 
\href{https://orcid.org/0000-0003-2002-984X}{\includegraphics[scale=0.06]{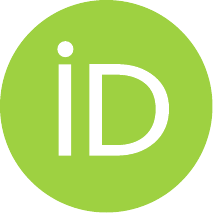}\hspace{1mm}Kang Liu}\thanks{Personal website: https://cruyffio.github.io/mysite/} \\
Independent Researcher\\
\texttt{liukangk11@gmail.com} \\
}

\renewcommand{\undertitle}{}

\hypersetup{
pdftitle={A template for the arxiv style},
pdfsubject={q-bio.NC, q-bio.QM},
pdfauthor={David S.~Hippocampus, Elias D.~Striatum},
pdfkeywords={First keyword, Second keyword, More},
}

\begin{document}
\maketitle

\vspace{-20pt}

\begin{abstract}

This paper investigates the moment monotonicity property of Weibull, Gamma, and Log-normal distributions. We provide the first complete mathematical proofs for the monotonicity of the function  $E(X^n)^{\frac{1}{n}}$ specific to these distributions. Through the derivations, we identify a key property: in many cases, one of the two parameters defining each distribution can be effectively canceled out. This finding opens up opportunities for improved parameter estimation of these random variables. Our results contribute to a deeper understanding of the behavior of these widely used distributions and offer valuable insights for applications in fields such as reliability engineering, econometrics, and machine learning.

\end{abstract}



\section{Introduction}
Weibull, Gamma, and Log-normal distributions are widely used in reliability engineering \citep{zio2009reliability,kapur2014reliability,xu2021machine}, econometrics \citep{imbens2009recent,dougherty2011introduction,anselin2022spatial,hansen2022econometrics}, medical statistics \citep{bland2015introduction,abiri2017tensile}, production system engineering \citep{alavian2018alpha,alavian2019alpha,alavian2020alpha,eun2022production,alavian2022alpha,liu2021alpha}, robotics \citep{cheng2011reliability,li2019imprecise,liu2019vehicle}, machine learning \citep{von2022knowledge}, and even natural language processing \citep{stanisz2024complex,liu2024setcse}. These distributions model a broad range of random variables, particularly those related to time-to-failure or time-between-events. Examples include the uptime of a machine in a production system \citep{li2008production} and the arrival time of a bus at a station \citep{yu2011bus}.

In probability theory \citep{renyi2007probability,laha2020probability}, the moments of a probability distribution provide quantitative measures of its shape. For instance, the first moment corresponds to the expected value, the second central moment to variance, the third standardized moment to skewness, and the fourth standardized moment to kurtosis. Suppose $X$ is a random variable following probability distribution function $f(x)$, its $n$-th moment is given by the following integral:
\begin{equation}
    E(X^n) = \int^{\infty}_{-\infty} x^n f(x) dx \text{ .}
\end{equation}


In this paper, we establish moment monotonicity theorems for Weibull, Gamma, and Log-normal distributions, demonstrating the monotonicity property of the function $E(X^n)^{\frac{1}{n}}$. To the best of our knowledge, this is the first paper to present complete mathematical proofs for these three distributions. Through these proofs, we observe that $E(X^n)^{\frac{1}{n}}$ exhibits specific properties that enable the derivation of its monotonicity. These results not only deepen our understanding of moment behavior in these distributions but also facilitate further research, such as the estimation of distribution parameters \citep{nielsen2011parameter,gomes2008parameter,ginos2009parameter}.

    
    

\section{Background}

It is important to note that a general form of the moment monotonicity property \citep{mcgill2022} can be proven using Jensen's inequality \citep{lin2010probability,boyd2004convex}. However, to the best of our knowledge, mathematical proofs for each specific probability density function are not yet available in the literature. These proofs are crucial for identifying particular patterns within these distributions, which can significantly benefit related areas of research, such as parameter estimation for these distributions.

\section{Weibull, Gamma and Log-normal Distribution and their Moments}

The probability density function of a Weibull random variable is given by the following formula:
\begin{equation}\label{eq:weibull_pdf}
f(x; k, \lambda)= \begin{cases}
\frac{k}{\lambda} \left(\frac{x}{\lambda}\right)^{k-1}  e^{-\left(\frac{x}{\lambda}\right)^k} &\text{ for } x\geq 0 \\
0 &\text{ for } x<0
\end{cases}
\; \text{ ,}
\end{equation}
where $k > 0$ is the shape parameter and $\lambda > 0$ is the scale parameter of the distribution. The probability density function is shown in Figure \ref{fig:weibull_pdf} for different values of $k$ and $\lambda$. In addition, the $i$-th moment of Weibull distribution is:
\begin{equation}
E(X^i) = \lambda^i \Gamma\left(1 + \frac{i}{k}\right) \text{ .}
\end{equation}

\begin{figure}
\centering

\begin{subfigure}{0.45\textwidth}
  \centering
  \includegraphics[width=0.8\linewidth]{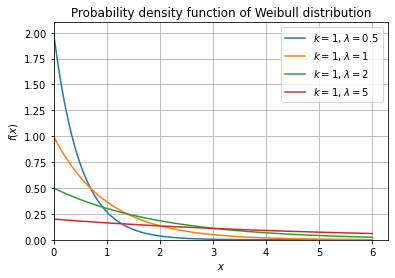}
  \caption{$k=1$}
  \label{fig:weibull_pdf1}
\end{subfigure}%
\begin{subfigure}{0.45\textwidth}
  \centering
  \includegraphics[width=0.8\linewidth]{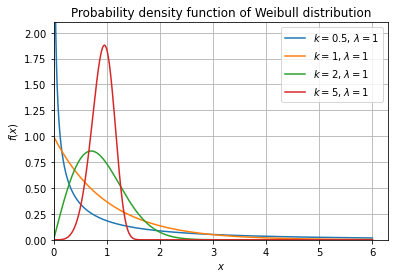}
  \caption{$\lambda=1$}
  \label{fig:weibull_pdf2}
\end{subfigure}

\caption{Probability density function of Weibull distribution.}
\label{fig:weibull_pdf}
\end{figure}

The probability density function of a Gamma random variable is given by the following formula:
\begin{equation}\label{eq:gamma_pdf}
f(x; k, \lambda)= \begin{cases}
\frac{1}{\beta^{\alpha} \Gamma(\alpha)} x^{\alpha-1} e^{\frac{-x}{\beta}} &\text{ for } x \geq 0
 \\
0 &\text{ for } x<0
\end{cases}
\; \text{ ,}
\end{equation}
where $\alpha > 0$ is the shape parameter and $\beta > 0$ is the scale parameter of the distribution. The probability density function is shown in Figure \ref{fig:gamma_pdf} for different values of $\sigma$ and $\mu$. In addition, the $i$-th moment of Gamma distribution is:
\begin{equation}\label{eq:moment_gamma}
E\left(X^i\right) = \frac{\beta^i \Gamma(i+\alpha)}{\Gamma(\alpha)} \text{ .}
\end{equation}

\begin{figure}
\centering

\begin{subfigure}{0.45\textwidth}
  \centering
  \includegraphics[width=0.8\linewidth]{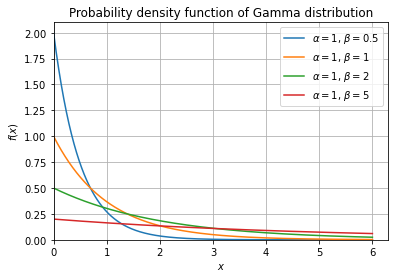}
  \caption{$\alpha=1$}
  \label{fig:gamma_pdf1}
\end{subfigure}%
\begin{subfigure}{0.45\textwidth}
  \centering
  \includegraphics[width=0.8\linewidth]{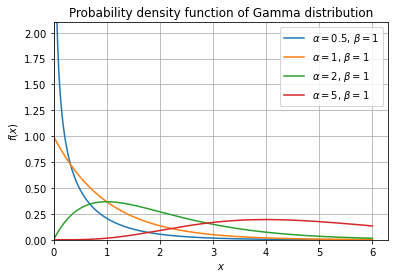}
  \caption{$\beta=1$}
  \label{fig:gamma_pdf2}
\end{subfigure}

\caption{Probability density function of Gamma distribution.}
\label{fig:gamma_pdf}
\end{figure}

The probability density function of a Log-normal random variable is given by the following formula:
\begin{equation}\label{eq:lognor_pdf}
f(x; k, \lambda)= \begin{cases}
\frac{1}{\sqrt{2\pi}\sigma x} e^{-\frac{(\ln{x} - \mu)^2}{2\sigma^2}} &\text{ for } x \geq 0
 \\
0 &\text{ for } x<0
\end{cases}
\; \text{ ,}
\end{equation}
where $\mu \in \mathbb{R}$ denotes the logarithm of location, and $\sigma > 0$ denotes logarithm of scale. The probability density function is shown in Figure \ref{fig:lognor_pdf} for different values of $\sigma$ and $\mu$. In addition, the $i$-th moment of Log-normal distribution is:
\begin{equation}
E\left(X^i\right) = e^{\mu i + \frac{1}{2} \sigma^2 i^2 } \text{ .}
\end{equation}

\begin{figure}
\centering

\begin{subfigure}{0.45\textwidth}
  \centering
  \includegraphics[width=0.8\linewidth]{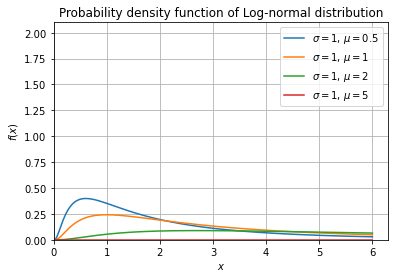}
  \caption{$\sigma=1$}
  \label{fig:lognor_pdf1}
\end{subfigure}%
\begin{subfigure}{0.45\textwidth}
  \centering
  \includegraphics[width=0.8\linewidth]{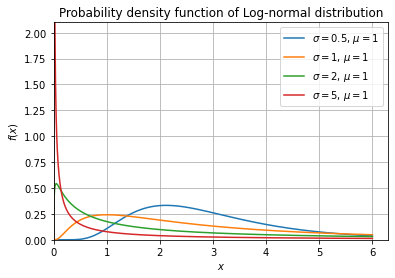}
  \caption{$\mu=1$}
  \label{fig:lognor_pdf2}
\end{subfigure}

\caption{Probability density function of Log-normal distribution.}
\label{fig:lognor_pdf}
\end{figure}

\section{Moment Monotonicity}
\label{sec:thm}

In this section, we state and prove the moment monotonicity property for each distribution of our interest.

\subsection{Moment Monotonicity of Weibull Distribution}

The following theorem states the moment monotonicity property of Weibull distribution.

\begin{thm}\label{thm:weibull}
Suppose random variable $X$ follows Weibull distribution, and $E(X^i)$ denotes the $i$-th moment of $X$. Then the random variable $X$ satisfy the following inequality: 
\begin{equation}\label{eq:moments}
    E(X^n)^{\frac{1}{n}} \geq E(X^m)^{\frac{1}{m}},
\end{equation}
where $n > m$.
\end{thm}

\begin{proof}
Suppose random variable $X$ follows Weibull distribution, where the scale parameter and shape parameter are denoted as $\lambda$ and $k$, respectively. By definition, the two parameters satisfy $\lambda\in(0,+\infty)$ and $k\in(0, +\infty)$.

The $i$-th moment of $X$ is given by:
\begin{equation}
    E(X^i) = \lambda^i \Gamma\left(1 + \frac{i}{k}\right),
\end{equation}
where
\begin{equation}
\Gamma(x) = \int^{\infty}_{0} t^{x-1}e^{-t}dt.
\end{equation}

Thus, for any positive integers $n>m$, we have $E(X^n) = \lambda^n \Gamma\left(1 + \frac{n}{k}\right)$ and $E(X^m) = \lambda^m \Gamma\left(1 + \frac{m}{k}\right)$. 

Let
\begin{equation}
    R = \frac{E(X^n)^m}{E(X^m)^n} = \frac{\lambda^{mn}\Gamma^m\left(1 + \frac{n}{k}\right)}{\lambda^{mn}\Gamma^n\left(1 + \frac{m}{k}\right)} = \frac{\Gamma^m\left(1+\frac{n}{k}\right)}{\Gamma^n\left(1+\frac{m}{k}\right)}.
\end{equation}

The derivative $\frac{dR}{dk}$ is given by:
\begin{equation}\label{eq:dRdk}
\begin{aligned}
\frac{dR}{dk} =& \frac{\Gamma^n(1+\frac{m}{k})mn(-k^{-2})\Gamma^{m-1}\left(1+\frac{n}{k}\right) \Gamma'\left(1+\frac{n}{k}\right) - \Gamma^m\left(1+\frac{n}{k}\right)mn(-k^{-2})\Gamma^{n-1}\left(1+\frac{m}{k}\right)\Gamma'\left(1+\frac{m}{k}\right)}{\Gamma^{2n}\left(1+\frac{m}{k}\right)}\\
=& \frac{mn \Gamma^{n-1}\left(1+\frac{m}{k}\right)\Gamma^{m-1}\left(1+\frac{n}{k}\right)\bigg(\Gamma\left(1+\frac{n}{k}\right)\Gamma'\left(1+\frac{m}{k}\right) - \Gamma\left(1+\frac{m}{k}\right)\Gamma'\left(1+\frac{n}{k}\right) \bigg)}{k^2\Gamma^{2n}\left(1+\frac{m}{k}\right)}.
\end{aligned}
\end{equation}

As one can see, the denominator of \eqref{eq:dRdk} is positive, to show \eqref{eq:dRdk} is negative, we consider what is referred to as digamma function, $\psi(x)$, and its derivative, $\psi'(x)$:
\begin{equation}
\begin{gathered}
\psi(x) = \frac{\Gamma'(x)}{\Gamma(x)},\\
\psi'(x) = -\int^{1}_{0} \frac{t^{x-1}}{1-t}\ln{t}dt > 0  \text{ for } x > 0.
\end{gathered}
\end{equation}

Thus, for $x>0$, $\psi(x)$ is monotonically increasing, i.e., for $n>m$, $\psi(1 + \frac{n}{\beta}) > \psi(1 + \frac{m}{\beta})$. In other words,

\begin{equation}
\begin{gathered}
\frac{\Gamma'\left(1+\frac{n}{k}\right)}{\Gamma\left(1+\frac{n}{k}\right)}> \frac{\Gamma'\left(1+\frac{m}{k}\right)}{\Gamma\left(1+\frac{m}{k}\right)}\\
\iff \Gamma\left(1+\frac{m}{k}\right)\Gamma'\left(1+\frac{n}{k}\right) > \Gamma\left(1+\frac{n}{k}\right)\Gamma'\left(1+\frac{m}{k}\right)\\
\iff \Gamma\left(1+\frac{n}{k}\right)\Gamma'\left(1+\frac{m}{k}\right) - \Gamma\left(1+\frac{m}{k}\right)\Gamma'\left(1+\frac{n}{k}\right) < 0 .
\end{gathered}
\end{equation}

Therefore, $\frac{dR}{dk}< 0$, i.e., $R$ is decreasing with respect to $k$. Since $k\in (0, +\infty)$, and $\lim\limits_{k\to \infty} R = \frac{\Gamma^m(1)}{\Gamma^n(1)} = 1$. Thus, we have $R \geq 1$, in other words,
\begin{equation}
    E(X^n)^{m} \geq E(X^m)^{n} \Longrightarrow E(X^n)^{\frac{1}{n}} \geq E(X^m)^{\frac{1}{m}}.
\end{equation}

\end{proof}

\subsection{Moment Monotonicity of Gamma Distribution}

The following theorem states the moment monotonicity property of Gamma distribution.

\begin{thm}\label{thm:gamma}
Suppose random variable $X$ follows Gamma distribution, and $E(X^i)$ denotes the $i$-th moment of $X$. Then the random variable $X$ satisfy the following inequality: 
\begin{equation}\label{eq:moments_gamma}
    E(X^n)^{\frac{1}{n}} \geq E(X^m)^{\frac{1}{m}},
\end{equation}
where $n > m$.
\end{thm}




\begin{proof}
As given by \eqref{eq:moment_gamma}, the expression of $E(X^i)$ is
\begin{equation}
    E\left(X^i\right) = \frac{\beta^i \Gamma(i+\alpha)}{\Gamma(\alpha)} \text{ .}
\end{equation}
Let $R = \frac{E(X^n)^m}{E(X^m)^n}$, we have:
\begin{equation}
    R = \left( \frac{\beta^n \Gamma(n + \alpha)}{\Gamma(\alpha)}\right)^m \left( \frac{\beta^m \Gamma(m + \alpha)}{\Gamma(\alpha)}\right)^{-n} \text{.} 
\end{equation}

The derivative $\frac{dR}{dk}$ is given by:
\begin{equation}\label{eq:dRdk_gamma}
\begin{aligned}
\frac{dR}{dk} = m \psi(n + \alpha) - n \psi(m + \alpha) + (n - m) \psi(\alpha) < 0.
\end{aligned}
\end{equation}

Thus, $R$ is monotonically decreasing as $\alpha \rightarrow \infty$. To analyze the limit of $R$, we use Stirling’s approximation, which states that for large $\alpha$
\begin{equation}
    \Gamma(\alpha) \approx \sqrt{2\pi} \alpha^{\alpha - \frac{1}{2}} e^{-\alpha}.
\end{equation}

Applying this to each Gamma function:

\begin{equation}
\begin{aligned}
\Gamma(n + \alpha) \approx & \sqrt{2\pi} (\alpha+n)^{\alpha+n - \frac{1}{2}} e^{-(\alpha+n)}, \\
\Gamma(m + \alpha) \approx & \sqrt{2\pi} (\alpha+m)^{\alpha+m - \frac{1}{2}} e^{-(\alpha+m)}, \\
\Gamma(\alpha) \approx & \sqrt{2\pi} \alpha^{\alpha - \frac{1}{2}} e^{-\alpha}.
\end{aligned}
\end{equation}

Substituting these into the expression for $R$:

\begin{equation}
R \approx \frac{\left[ \sqrt{2\pi} (\alpha+n)^{\alpha+n-\frac{1}{2}} e^{-(\alpha+n)} \right]^m \cdot \left[ \sqrt{2\pi} \alpha^{\alpha-\frac{1}{2}} e^{-\alpha} \right]^{(n-m)}}
{\left[ \sqrt{2\pi} (\alpha+m)^{\alpha+m-\frac{1}{2}} e^{-(\alpha+m)} \right]^n}.
\end{equation}

Expanding the exponents:
\begin{equation}
R \approx \frac{(2\pi)^{\frac{m}{2}} (\alpha+n)^{m(\alpha+n-\frac{1}{2})} e^{-m(\alpha+n)} (2\pi)^{\frac{n-m}{2}} \alpha^{(n-m)(\alpha-\frac{1}{2})} e^{-(n-m)\alpha}}
{(2\pi)^{\frac{n}{2}} (\alpha+m)^{n(\alpha+m-\frac{1}{2})} e^{-n(\alpha+m)}}.
\end{equation}

Canceling $2\pi$ terms:

\begin{equation}
R \approx \frac{(\alpha+n)^{m(\alpha+n-\frac{1}{2})} e^{-m(\alpha+n)} \alpha^{(n-m)(\alpha-\frac{1}{2})} e^{-(n-m)\alpha}}
{(\alpha+m)^{n(\alpha+m-\frac{1}{2})} e^{-n(\alpha+m)}}.
\end{equation}

Rewriting exponents:
\begin{equation}
\begin{aligned}
& (\alpha+n)^{m(\alpha+n)} e^{-m(\alpha+n)} \alpha^{(n-m)\alpha} e^{-(n-m)\alpha}, \\
& (\alpha+m)^{n(\alpha+m)} e^{-n(\alpha+m)}.
\end{aligned}
\end{equation}

For large $\alpha$, using $(\alpha+a) \approx \alpha(1 + \frac{a}{\alpha})$, we approximate:
\begin{equation}
(\alpha+n)^{\alpha+n} \approx \alpha^{\alpha+n} e^n, \quad (\alpha+m)^{\alpha+m} \approx \alpha^{\alpha+m} e^m.
\end{equation}

Thus,
\begin{equation}
(\alpha+n)^{m(\alpha+n)} \approx \alpha^{m(\alpha+n)} e^{mn}, \quad (\alpha+m)^{n(\alpha+m)} \approx \alpha^{n(\alpha+m)} e^{nm}.
\end{equation}

Substituting:
\begin{equation}
R \approx \frac{\alpha^{m(\alpha+n)} e^{mn} \cdot \alpha^{(n-m)\alpha} e^{-(n-m)\alpha}}{\alpha^{n(\alpha+m)} e^{nm}}.
\end{equation}

Simplifying the exponent of $\alpha$:
\begin{equation}
m(\alpha+n) + (n-m)\alpha - n(\alpha+m) = m\alpha + mn + n\alpha - m\alpha - n\alpha - nm = 0.
\end{equation}

Thus, the powers of $\alpha$ cancel, and the exponentials cancel:

\begin{equation}
R \approx e^{mn - (n-m)\alpha - nm + n\alpha + nm} = e^0 = 1.
\end{equation}

Therefore, we have $R \geq 1$, in other words,
\begin{equation}
    E(X^n)^{m} \geq E(X^m)^{n} \Longrightarrow E(X^n)^{\frac{1}{n}} \geq E(X^m)^{\frac{1}{m}}.
\end{equation}

\end{proof}

\subsection{Moment Monotonicity of Log-normal Distribution}

The following theorem states the moment monotonicity property of Log-normal distribution.

\begin{thm}\label{thm:lognormal}
Suppose random variable $x$ follows Log-normal distribution, and $E(X^i)$ denotes the $i$-th moment of $X$. Then the random variable $X$ satisfy the following inequality: 
\begin{equation}\label{eq:moments_logn}
    E(X^n)^{\frac{1}{n}} \geq E(X^m)^{\frac{1}{m}},
\end{equation}
where $n > m$.
\end{thm}

\begin{proof}
The $i$-th moment of X is given by
\begin{equation}
    E(X^i) = e^{\mu i + \frac{1}{2} \sigma^2 i^2},
\end{equation}
thus, we have
\begin{equation}
    E(X^i)^{\frac{1}{i}}  =  e^{\mu + \frac{1}{2} \sigma^2 i}.
\end{equation}

Denote $G = \frac{E(X^n)^{\frac{1}{n}}}{E(X^m)^{\frac{1}{m}}}$, we obtain
\begin{equation}
    G = \frac{e^{\mu + \frac{1}{2} \sigma^2 n}}{e^{\mu + \frac{1}{2} \sigma^2 m}} = e^{\frac{1}{2}\sigma^2 (n-m)}.
\end{equation}
Since $n>m \Rightarrow \frac{1}{2}\sigma^2(n-m)$, we have $G \geq 1$, in other words,
\begin{equation}
    E(X^n)^{\frac{1}{n}} \geq E(X^m)^{\frac{1}{m}}.
\end{equation}

\end{proof}

\section{Conclusion and Future Work}


In this paper, we have explored the moment monotonicity property of the Weibull, Gamma, and Log-normal distributions, providing the first complete mathematical proofs for the monotonicity of the function $E(X^n)^{\frac{1}{n}}$ specific to these distributions. Through the derivation of these proofs, we have uncovered a noteworthy phenomenon: in many cases, one of the two parameters that define each distribution can be canceled out. This insight is not only a theoretical contribution but also has practical implications for the estimation of distribution parameters. The ability to simplify these parameters could lead to more efficient methods for parameter estimation, particularly in cases where accurate modeling of random variables is crucial.


Looking forward, the potential for leveraging this property in the context of parameter estimation remains an area of significant interest, including its application in real-world scenarios and across different domains. By refining these methods, we aim to provide new tools for practitioners who rely on these distributions to model complex systems and processes.

\bibliographystyle{unsrtnat}
\bibliography{references}  






\end{document}